\documentclass[12pt,draft]{article}
\usepackage{amsmath,amscd}
\usepackage{amssymb,latexsym,amsthm}
\usepackage{color}
\usepackage{amssymb}
\usepackage{color}
\usepackage{amsmath,amsthm,amscd}
\usepackage[latin1]{inputenc}
\usepackage{lscape}
\usepackage{fancyhdr}
\usepackage{amsfonts}
\usepackage{pb-diagram}

\numberwithin{equation}{section}
\newtheorem{theorem}{Theorem}[section]

\newtheorem{proposition}[theorem]{Proposition}

\newtheorem{corollary}[theorem]{Corollary}

\theoremstyle{definition}
\newtheorem{example}[theorem]{Example}
\newtheorem{definition}[theorem]{Definition}
\newtheorem{examples}[theorem]{Examples}

\newtheorem{remark}[theorem]{Remark}

\newtheorem{classification}[theorem]{Classification}

\newcommand{\cU}{\mbox{${\cal U}$}}

\title{\textbf{Koszulity for skew PBW extensions over fields}}
\author{H\'ector Su\'arez\footnote{Seminario de \'Algebra Constructiva - $SAC^2$, Departamento de Matem\'aticas,
Universidad Nacional de Colombia - sede Bogot\'a.} \footnote{Escuela de Matem\'aticas y Estad\'{\i}stica,
Universidad Pedag\'ogica y Tecnol\'ogica de Colombia - sede Tunja. } \\Armando Reyes \footnote{Seminario de \'Algebra Constructiva - $SAC^2$, Departamento de Matem\'aticas,
Universidad Nacional de Colombia - sede Bogot\'a.}}
\date{}
\begin{document}
\maketitle
\begin{abstract}
\noindent Koszul and homogeneous Koszul algebras were defined by Priddy in \cite{Priddy1970}. There
exist some relations between these algebras and the skew PBW  extensions introduced in
\cite{LezamaGallego}. In this paper we give conditions to guarantee that skew PBW extensions over
fields are Koszul or homogeneous Koszul. We also show that a constant skew PBW extension of a
field is a PBW deformation of its homogeneous version.

\bigskip


\bigskip

\end{abstract}

\section{Introduction}

Koszul and homogeneous Koszul algebras were introduced by Priddy in \cite{Priddy1970}; despite of
that these type of algebras have not been enough studied, they have important applications in
algebraic geo\-me\-try, Lie theory, quantum groups, algebraic topology and combinatorics. The
structure and history of Koszul homogeneous algebras were  detailed in \cite{Polishchuk}. There
exist nu\-me\-rous equivalent definitions of homogeneous Koszul algebras (see for example
\cite{Backelin}); in addition, Koszul algebras have been defined in a more general way by some
authors and they are commonly called \textquotedblleft Ge\-ne\-ra\-li\-zed Koszul
algebras\textquotedblright\ (see for example \cite{BeilinsonGinzburgSoerge1996},
\cite{Cassidy2008}, \cite{Li2012}, \cite{Woodcock1998}). In this paper  we work with the definition
given by Priddy, which is commonly called \textquotedblleft Koszul algebras in the classical
sense\textquotedblright.

Skew PBW extensions or $\sigma$-PBW extensions were defined in \cite{LezamaGallego}. Several
properties of these extensions have been recently studied (see for example \cite{lezamaore},
\cite{ArtamonovDerivations}, \cite{Gallego4}, \cite{Jimenez2}, \cite{LezamaAcostaReyes2015},
\cite{LezamaReyes},  \cite{Reyes2013}, \cite{Reyes2014}, \cite{Reyes2014UIS}, \cite{Reyes2015},
\cite{SuarezLezamaReyes2015}, \cite{Venegas2015}). There exist some relations between Koszul and
homogeneous Koszul  algebras with the skew PBW extensions. Our   interest in this article is to
study the Koszul property for the skew PBW extensions over fields. For this purpose we  classify
the skew PBW extensions in five sub-classes: constant, bijective, pre-commutative,
quasi-commutative and semi-commutative, and  we show that a skew PBW extension $A$ of a field  is
Koszul (homogeneous Koszul) when $A$ is pre-commutative and constant (semi-commutative). Finally,
following the ideas presented in \cite{Braverman}, we show that a constant skew PBW extension of a
field is a PBW deformation of  its homogeneous version.

\section{Skew PBW extensions}\label{section2}
In this section we recall some elementary properties of skew PBW extensions; in addition, we will
introduce some sub-classes of them: constant, pre-commutative and semi-commutative. Examples of
these sub-classes are presented.

\subsection{Definitions and properties}

\begin{definition}[\cite{LezamaGallego}, Definition 1]\label{def.skewpbwextensions}
Let $R$ and $A$ be rings. We say that $A$ is a \textit{skew PBW
extension of} $R$ (also called a $\sigma$-PBW extension
of $R$) if the following conditions hold:
\begin{enumerate}
\item[\rm (i)]$R\subseteq A$;
\item[\rm (ii)]there exist finitely many elements $x_1,\dots ,x_n\in A$ such $A$ is a left $R$-free module with basis
\begin{center}
${\rm Mon}(A):= \{x^{\alpha}=x_1^{\alpha_1}\cdots x_n^{\alpha_n}\mid \alpha=(\alpha_1,\dots
,\alpha_n)\in \mathbb{N}^n\}$, with $\mathbb{N}:=\{0,1,2,\dots\}$.
\end{center}
The set $Mon(A)$ is called the set of standard monomials of $A$.

\item[\rm (iii)]For each $1\leq i\leq n$ and any $r\in R\ \backslash\ \{0\}$, there exists an element $c_{i,r}\in R\ \backslash\ \{0\}$ such that
\begin{equation}\label{sigmadefinicion1}
x_ir-c_{i,r}x_i\in R.
\end{equation}
\item[\rm (iv)]For any elements $1\leq i,j\leq n$ there exists $c_{i,j}\in R\ \backslash\ \{0\}$ such that
\begin{equation}\label{sigmadefinicion2}
x_jx_i-c_{i,j}x_ix_j\in R+Rx_1+\cdots +Rx_n.
\end{equation}
Under these conditions we will write $A:=\sigma(R)\langle
x_1,\dots,x_n\rangle$.
\end{enumerate}
\end{definition}

The notation $\sigma(R)\langle x_1,\dots,x_n\rangle$ and the name of the skew PBW extensions are due to the following proposition.

\begin{proposition}[\cite{LezamaGallego}, Proposition 3]\label{sigmadefinition}
Let $A$ be a skew PBW extension of $R$. For each $1\leq i\leq
n$, there exists an injective endomorphism $\sigma_i:R\rightarrow
R$ and a $\sigma_i$-derivation $\delta_i:R\rightarrow R$ such that
\begin{equation}
x_ir=\sigma_i(r)x_i+\delta_i(r),\ \ \ \  \ r \in R.
\end{equation}
\end{proposition}

\begin{definition}\label{sigmapbwderivationtype}
Let $A$ be a skew PBW extension of $R$, $\Sigma:=\{\sigma_1,\dotsc, \sigma_n\}$ and $\Delta:=\{\delta_1,\dotsc, \delta_n\}$, where $\sigma_i$ and $\delta_i$ ($1\leq i\leq n$) are as in the Proposition \ref{sigmadefinition}.
\begin{enumerate}
\item[\rm (a)]  $A$ is called {\em pre-commutative} if the condition  {\rm(}iv{\rm)} in Definition
\ref{def.skewpbwextensions} is replaced by:\\
 For any $1\leq i,j\leq n$ there exists $c_{i,j}\in R\ \backslash\ \{0\}$ such that
\begin{equation}\label{relat.pre-comm}
x_jx_i-c_{i,j}x_ix_j\in Rx_1+\cdots +Rx_n.
\end{equation}

\item[\rm (b)]\label{def.quasicom} $A$ is called \textit{quasi-commutative} if the conditions
{\rm(}iii{\rm)} and {\rm(}iv{\rm)} in Definition
\ref{def.skewpbwextensions} are replaced by \begin{enumerate}
\item[\rm (iii')] for each $1\leq i\leq n$ and all $r\in R\ \backslash\ \{0\}$ there exists $c_{i,r}\in R\ \backslash\ \{0\}$ such that
\begin{equation}
x_ir=c_{i,r}x_i;
\end{equation}
\item[\rm (iv')]for any $1\leq i,j\leq n$ there exists $c_{i,j}\in R\ \backslash\ \{0\}$ such that
\begin{equation}
x_jx_i=c_{i,j}x_ix_j.
\end{equation}
\end{enumerate}
\item[\rm (c)]  $A$ is called \textit{bijective} if $\sigma_i$ is bijective for each $\sigma_i\in \Sigma$, and $c_{i,j}$ is invertible for any $1\leq
i<j\leq n$.
\item[\rm (d)]   Any element $r$ of $R$ such that $\sigma_i(r)=r$ and $\delta_i(r)=0$ for all $1\leq i\leq n$ will be called a {\em constant}. $A$ is called \emph{constant} if every element of $R$ is constant.
\item[\rm (e)]  $A$ is called {\em semi-commutative} if $A$ is quasi-commutative and constant.
\end{enumerate}
\end{definition}

Recall that  a \textit{filtered ring} is a ring $B$ with a family $F(B)=\{F_n(B)\mid n \in
\mathbb{Z}\}$ of subgroups of the additive group of $B$ where we have the ascending chain $\dotsb
\subset F_{n-1}(B)\subset F_n(B)\subset \dotsb$ such that $1\in F_0(B)$ and $F_n(B)F_m(B)\subseteq
F_{n+m}(B)$ for all $n,m\in \mathbb{Z}$. From a filtered ring $B$  it is possible to construct its
associated graded ring $Gr(B)$ taking  $Gr(B)_n:=F_n(B)/F_{n-1}(B)$. The following proposition
establishes that one can construct a quasi-commutative skew PBW extension from a given skew PBW
extension of a ring $R$.

\begin{proposition}[\cite{LezamaReyes},  Proposition 2.1]\label{remarkAsigma}  Let $A$ be a skew PBW extension of $R$. Then, there exists a quasicommutative
skew PBW extension $A^{\sigma}$ of $R$ in $n$ variables $z_1,\dotsc, z_n$ defined by the relations $z_ir=c_{i,r}z_i$, $z_jz_i=c_{i,j}z_iz_j$, for $1\le i\le n$, where $c_{i,r}, c_{i,j}$ are the same constants that define $A$. Moreover, if $A$ is bijective then $A^{\sigma}$ is also bijective.
\end{proposition}

The next proposition  computes the graduation of a skew PBW extension.

\begin{theorem}[\cite{LezamaReyes}, Theorem 2.2]\label{teo.Gr(A)}
Let $A$ be an arbitrary skew PBW extension of $R$. Then, $A$ is
a filtered ring with increasing filtration given by
\begin{equation}\label{eq1.3.1a}
F_m(A):=\begin{cases} R & {\rm if}\ \ m=0\\ \{f\in A\mid {\rm deg}(f)\le m\} & {\rm if}\ \ m\ge 1
\end{cases}
\end{equation}
and the corresponding graded ring $Gr(A)$ is isomorphic to $A^{\sigma}$.
\end{theorem}

\subsection{Examples and classification}\label{ex and classif}

\begin{examples}\label{general exam}
In \cite{LezamaGallego} and \cite{LezamaReyes}  was presented a list of
examples of quasi-commutative or bijective skew PBW extensions. We also classify these examples according to
Definition  \ref{sigmapbwderivationtype}. Through this paper,
$\mathbb{K}$ will denote a field and $K$ a commutative ring.
\begin{enumerate}
\item Classical polynomial ring; Ore extensions  of bijective type and Weyl algebras;
Universal enveloping algebra of a Lie algebra; Tensor product; crossed product; Algebra of
$q$-differential operators; Algebra of shift operators;  Mixed algebras; Algebra of discrete linear
systems; Linear partial differential operators; Linear partial shift operators; Algebra of linear
partial difference operators; Algebra of linear partial q-dilation operators; Algebra of linear
partial q-differential operators;  Diffusion algebra 1 (\cite{Reyes2014UIS}); Diffusion algebra 2
(\cite{LezamaReyes}); Additive analogue of the Weyl algebra; Multiplicative analogue of the Weyl
algebras; Quantum algebras; Dispin algebras; Woronowicz algebras; Complex algebras; Algebra
$\textbf{U}$; Manin algebras;  Algebra of quantum matrices; $q$-Heisenberg algebras;  Quantum
enveloping algebras of $\mathfrak{sl}(2,\mathbb{K})$; The algebra of differential operators  on a
quantum space; Witten's deformation of   $\cU(\mathfrak{sl}(2,\mathbb{K}))$; Quantum Weyl algebra
of Maltsiniotis; Quantum Weyl algebras; Multiparameter quantized Weyl algebras; Quantum symplectic
space and Quadratic algebras in 3 variables.
\item\label{jordan plane}Jordan plane. The Jordan plane $A$ is the free $\mathbb{K}$-algebra generated by $x,y$ with relation $yx = xy + x^2$, so  $A=\mathbb{K}\langle x, y\rangle/ \langle yx-xy-x^2\rangle\cong \sigma(\mathbb{K}[x])\langle y\rangle$.
\item\label{Sklyanin} Particular  Sklyanin algebra. The  Sklyanin algebra (Example 1.14, \cite{Rogalski}) is the $\mathbb{K}$-algebra
$S = \mathbb{K}\langle x, y, z\rangle/\langle ayx + bxy + cz^2, axz + bzx + cy^2, azy + byz + cx^2\rangle$, where $a, b, c\in \mathbb{K}$. If $c\neq 0$ then $S$ is not a skew PBW extension. If $c=0$ and $a,b\neq 0$ then in $S$: $yx = -\frac{b}{a}xy$; $zx =  -\frac{a}{b}xz$ and $zy =  -\frac{b}{a}yz$, therefore $S\cong\sigma(\mathbb{K})\langle x,y,z\rangle$ is a skew PBW extension of $\mathbb{K}$, and we call this algebra a \emph{particular Sklyanin algebra}.
\item Multi-parameter quantum affine $n$-spaces.  Let $n\geq 1$ and $\textbf{q}$ be a  matrix $(q_{ij})_{n\times n}$ with entries in a field $\mathbb{K}$  where
 $q_{ii} = 1$ y $q_{ij}q_{ji} = 1$ for all $1\leq i, j \leq n$. Then multi-parameter quantum affine $n$-space $\mathcal{O}_{\textbf{q}}(\mathbb{K}^n)$ is defined  to be $\mathbb{K}-$algebra generated by
 $x_1,\dots, x_n$ with the relations $x_jx_i = q_{ij}x_ix_j$
for all $1\leq i, j \leq n$.
\item \label{ex.homog envoluniversal} Homogenized enveloping algebra (\cite{Smith2}, Chapter 12).
Let $\mathcal{G}$ a finite dimensional Lie algebra over $\mathbb{K}$ with basis $\{x_1, \dots,
x_n\}$ and $\mathcal{U}(\mathcal{G})$ its enveloping algebra. The \emph{homogenized enveloping
algebra} of $\mathcal{G}$ is $\mathcal{A}(\mathcal{G}):= T(\mathcal{G}\oplus \mathbb{K}z)/\langle
R\rangle$, where $T(\mathcal{G}\oplus \mathbb{K}z)$ is the tensor algebra, $z$ is a new variable,
and $R$ is spanned by $\{z\otimes x-x\otimes z\mid x\in \mathcal{G}\}\cup \{x\otimes y-y\otimes
x-[x,y]\otimes z\mid x,y\in \mathcal{G}\}$. From the PBW Theorem for $\mathcal{G}\otimes
\mathbb{K}(z)$, considered as a Lie algebra over $\mathbb{K}(z)$, we get that
$\mathcal{A}(\mathcal{G})$ is a skew PBW extension of $\mathbb{K}[z]$.
\end{enumerate}

\begin{classification}\label{classif}

We classify the above examples of skew PBW extensions as constant (C), bijective (B),
pre-commutative (P), quasi-commutative (QC) and semi-commutative (SC); the classification is
presented in the next table, where the symbols $\star$ and $\checkmark$ denote negation and
affirmation, respectively.

\begin{center}

\scriptsize{
\begin{tabular}{|l|c|c|c|c|c|}\hline
\multicolumn{6}{|c|}{\textbf{Classification of Examples \ref{general exam}  }}\\ \hline
 \textbf{Skew PBW extension} &   \textbf{C} & \textbf{B} & \textbf{P} & \textbf{QC} &  \textbf{SC} \\
\hline \hline  Classical polynomial ring   & $\checkmark$ & $\checkmark$ & $\checkmark$  & $\checkmark$ &  $\checkmark$\\ \hline
 Ore extensions  of bijective type  & $\star$ & $\checkmark$ &  $\checkmark$ & $\star$ & $\star$\\ \hline
 Weyl algebra  &  $\star$ & $\checkmark$  &  $\checkmark$ & $\star$ & $\star$ \\ \hline
 Jordan plane &  $\star$ & $\checkmark$ & $\checkmark$ &   $\star$ &  $\star$ \\ \hline
 Particular  Sklyanin algebra &  $\checkmark$ & $\checkmark$ & $\checkmark$ &  $\checkmark$ &  $\checkmark$ \\ \hline
 Universal enveloping algebra of a Lie algebra  & $\checkmark$ & $\checkmark$ &   $\checkmark$ & $\star$ & $\star$\\ \hline
 Homogenized enveloping algebra $\mathcal{A}(\mathcal{G})$ & $\checkmark$ & $\checkmark$ &   $\checkmark$ & $\star$ & $\star$\\ \hline
Tensor product  & $\checkmark$ &$\checkmark$ & $\checkmark$& $\star$ &$\star$ \\ \hline
 Crossed product  & $\star$ & $\checkmark$ & $\star$ &$\star$ & $\star$\\ \hline
 Algebra of $q$-differential operators   & $\star$ & $\checkmark$ & $\checkmark$ &$\star$ & $\star$\\ \hline
 Algebra of shift operators   & $\star$ & $\checkmark$ &$\checkmark$ &$\checkmark$ & $\star$ \\ \hline
 Mixed algebra  & $\star$& $\checkmark$ &$\star$ &$\star$ &$\star$ \\ \hline
 Algebra of discrete linear systems   & $\star$&$\checkmark$ &$\checkmark$ &$\checkmark$ & $\star$\\ \hline
 Linear partial differential operators   & $\star$& $\checkmark$& $\checkmark$&$\star$ &$\star$ \\ \hline
 Linear partial shift operators   &$\star$ &$\checkmark$ &$\checkmark$ & $\checkmark$&$\star$ \\ \hline
 Algebra of linear partial difference operators   & $\star$& $\checkmark$& $\checkmark$ &$\star$ & $\star$\\ \hline
 Algebra of linear partial q-dilation operators   & $\star$ &$\checkmark$ & $\checkmark$&$\checkmark$ &$\checkmark$ \\ \hline
 Algebra of linear partial q-differential operators   &$\star$ & $\checkmark$ & $\checkmark$ &$\star$ &$\star$ \\ \hline
 Diffusion algebra 1  & $\checkmark$ & $\checkmark$ & $\checkmark$ & $\star$ & $\star$\\ \hline
 Diffusion algebra 2  & $\checkmark$ & $\checkmark$ & $\checkmark$ & $\star$ & $\star$\\ \hline
 Additive analogue of the Weyl algebra   &$\checkmark$ &$\checkmark$ & $\star$& $\star$& $\star$\\ \hline
 Multiplicative analogue of the Weyl algebra  &$\checkmark$ &$\checkmark$ &$\checkmark$ & $\checkmark$& $\checkmark$\\ \hline
 Quantum algebra   &$\checkmark$ &$\checkmark$ &$\checkmark$ & $\star$& $\star$\\ \hline
 Dispin algebra  & $\checkmark$ & $\checkmark$ & $\checkmark$ & $\star$ & $\star$\\ \hline
 Woronowicz algebra  &$\checkmark$ &$\checkmark$ &$\checkmark$ & $\star$& $\star$ \\ \hline
 Complex algebra    & $\star$  &$\checkmark$ &$\star$  &$\star$  & $\star$ \\ \hline
 Algebra $\textbf{U}$   & $\star$  &$\checkmark$ &$\star$  &$\star$  & $\star$ \\ \hline
Manin algebra   & $\star$  &$\checkmark$ & $\checkmark$  &$\star$  & $\star$ \\ \hline
$q$-Heisenberg algebra  & $\checkmark$  &$\checkmark$ & $\checkmark$  &$\star$  & $\star$  \\ \hline
 Quantum enveloping algebra of $\mathfrak{sl}(2,\mathbb{K})$ & $\star$  &$\checkmark$ &$\star$  &$\star$  & $\star$ \\ \hline
 Hayashi's algebra  & $\star$  &$\checkmark$ &$\star$  &$\star$  & $\star$ \\ \hline
 Multi-parameter quantum affine $n$-space  &$\checkmark$  &$\checkmark$  &$\checkmark$  & $\checkmark$ & $\checkmark$ \\ \hline
 The algebra of differential operators on a quantum space $S_q$ & $\star$ & $\checkmark$  & $\star$ & $\star$ &$\star$ \\ \hline
Witten's deformation of   $\cU(\mathfrak{sl}(2,\mathbb{K}))$& $\star$ & $\checkmark$  & $\star$ & $\star$ &$\star$  \\ \hline
 Quantum Weyl algebra of Maltsiniotis  & $\star$ & $\checkmark$  & $\star$ & $\star$ &$\star$ \\ \hline
 Quantum Weyl algebra  & $\star$ & $\checkmark$  & $\star$ & $\star$ &$\star$  \\ \hline
 Multiparameter quantized Weyl algebra & $\star$ & $\checkmark$  & $\star$ & $\star$ &$\star$  \\ \hline
 Quantum symplectic space & $\star$ & $\checkmark$  & $\star$ & $\star$ &$\star$ \\ \hline
 Quadratic algebras in 3 variable & $\star$ & $\checkmark$  & $\star$ & $\star$ &$\star$  \\ \hline

\end{tabular}}
\end{center}
\end{classification}
\end{examples}

\begin{example}\label{sridharan} \emph{Sridharan enveloping algebra of 3-dimensional Lie algebra $\mathcal{G}$}.
Let $\mathcal{G}$ be a finite dimensional Lie algebra, and let $f \in Z^2(\mathcal{G},\mathbb{K})$
be an arbitrary \textit{$2-$cocycle}, that is, $f: \mathcal{G}\times\mathcal{G}\to
\mathbb{K} \text{ such that } f (x, x) =0$  and  $$f(x, [y, z])+ f ( z, [x, y])  + f ( y, [z, x] )=
0$$ for all $x, y, z\in
\mathcal{G}$. The Sridharan enveloping algebra of $\mathcal{G}$ is defined to be the
associative algebra $\mathcal{U}_f(\mathcal{G}) = T(\mathcal{G})/I,$ where  $T(\mathcal{G})$ is the tensor algebra of $\mathcal{G}$ and  $I$ is the two-sided
ideal of $T(\mathcal{G})$ generated by the elements
$$(x\otimes y) - (y \otimes x) - [x, y] - f(x, y), \text{ for all $x, y\in
\mathcal{G}$}.$$
Note that if $f=0$ then $\mathcal{U}_f(\mathcal{G})=\mathcal{U}_0(\mathcal{G})=\mathcal{U}(\mathcal{G})$. For $x\in \mathcal{G}$, we still denote by $x$ its image in
$\mathcal{U}_f(\mathcal{G})$. $\mathcal{U}_f(\mathcal{G})$
is a filtered algebra with the associated graded algebra
$gr(\mathcal{U}_f(\mathcal{G}))$ being a polynomial algebra. Let $\mathbb{K}$ be an algebraically closed field of characteristic zero.  If $\mathcal{G}$ is a Lie $\mathbb{K}$-algebra of dimension three, then the Sridharan  enveloping algebra $\mathcal{U}_f(\mathcal{G})$ for $f \in Z^2(\mathcal{G},\mathbb{K})$ is isomorphic to one of  ten associative $\mathbb{K}$-algebras (see \cite{Nuss1}, Theorem 1.3), which is defined by three generators $x,y,z$ and the commutation relations as the following table shows. Therefore, the Sridharan  enveloping algebra $\mathcal{U}_f(\mathcal{G})$ is a skew PBW extension of $\mathbb{K}$, i.e. $\mathcal{U}_f(\mathcal{G})\cong\sigma(\mathbb{K})\langle x,y,z\rangle$, and it is classified as follows:

\begin{center}
\scriptsize{
\begin{tabular}{|c|c|c|c|c|c|c|c|c|}\hline
\multicolumn{9}{|c|}{\textbf{ Sridharan enveloping algebra of 3-dimensional Lie algebra $\mathcal{G}$}}\\ \hline
\textbf{Type} & $[x,y]$ & $[y,z]$ & $[z,x]$ & \textbf{C} & \textbf{B} & \textbf{P} & \textbf{QC} &  \textbf{SC}\\ \hline\hline

 1 & $0$ & $0$ & $0$ & $\checkmark$  & $\checkmark$  & $\checkmark$  & $\checkmark$  & $\checkmark$   \\ \hline
 2 & $0$ & $x$ & $0$ & $\checkmark$ & $\checkmark$ & $\checkmark$  & $\star$  & $\star$ \\ \hline
 3& $x$ & $0$ & $0$ & $\checkmark$ & $\checkmark$ & $\checkmark$  & $\star$  & $\star$ \\ \hline
4 & $0$ & $\alpha y$ & $-x$ & $\checkmark$ & $\checkmark$ & $\checkmark$  & $\star$  & $\star$ \\ \hline
5 & $0$ & $y$ & $-(x+y)$ & $\checkmark$ & $\checkmark$ & $\checkmark$  & $\star$  & $\star$ \\ \hline
 6 & $z$ & $-2y$ & $-2x$ & $\checkmark$ & $\checkmark$ & $\checkmark$  & $\star$  & $\star$ \\ \hline
 7& $1$ & $0$ & $0$ & $\checkmark$ & $\checkmark$ & $\star$   & $\star$  & $\star$ \\ \hline
  8& $1$ & $x$ & $0$ & $\checkmark$ & $\checkmark$ & $\star$   & $\star$  & $\star$ \\ \hline
  9& $x$ & $1$ & $0$ &  $\checkmark$ & $\checkmark$ & $\star$   & $\star$  & $\star$ \\ \hline
   10& $1$ & $y$ & $x$ &  $\checkmark$ & $\checkmark$ & $\star$   & $\star$  & $\star$ \\ \hline
\end{tabular}}
\end{center}
where $\alpha\in \mathbb{K} \backslash\{0\}$.

\end{example}

\section{Koszulity}

Some authors have defined Koszul algebras in a more general sense than \cite{Priddy1970} (see for
example \cite{BeilinsonGinzburgSoerge1996}, \cite{Li2012}, \cite{Woodcock1998}).  Our focus is to
study the Koszul property for skew PBW extensions taking into account the definition given  in
\cite{Priddy1970}. In this section we give sufficient conditions to guarantee that skew PBW
extensions are Koszul or homogeneous Koszul. For this purpose, we show the relationship between
$\mathbb{K}$-algebras that are skew PBW extensions and certain classes of algebras defined in
\cite{Priddy1970} containing the Koszul and homogeneous Koszul algebras. Let $L:=\mathbb{K}\langle x_1,\dots, x_n\rangle$  the free associative algebra (tensor algebra) in $n$ generators $x_1,\dots, x_n$. Note that $L$ is positively graded with graduation given by $L:=\bigoplus_{j\geq 0}L_j$, where
$L_0= \mathbb{K}$ and $L_j$ spanned by all words of length $j$ in the alphabet $\{x_1, \dots,
x_n\}$, for $j>0$.

\subsection{Pre-Koszul algebras}

We present a definition of pre-Koszul and homogeneous pre-Koszul algebras, analogous to the definition given by Priddy in \cite{Priddy1970}.

\begin{definition}\label{def.varios} Let $L=\mathbb{K}\langle x_1,\dots, x_n\rangle$ and let $B:=L/I$.
\begin{enumerate}
\item[\rm (i)] $B$ is said to be a \emph{pre-Koszul} algebra if  $I$ is a two sided ideal generated by elements of the form
\begin{equation}\label{rep.prekosz}
\sum_{i=1}^{n}c_ix_i + \sum_{1\leq j,k\leq n}c_{j,k}x_jx_k, \text{ where } c_i \text{ and } c_{j,k} \text{ are in } \mathbb{K},
\end{equation}
\item[\rm (ii)] A pre-Koszul algebra is said to be \emph{pre-Koszul homogeneous} if $c_i=0$, for $1\leq i\leq n$ in  (\ref{rep.prekosz}).
\end{enumerate}
\end{definition}

Presentations of special types of skew PBW extensions are given in the following remark.

\begin{remark}\label{notac.skewpbwext} Let $A= \sigma(\mathbb{K})\langle x_1,\dots, x_n\rangle$ be a skew PBW extension of a field $\mathbb{K}$.
\begin{enumerate}
\item  We note that $A = \mathbb{K}\langle x_1,\dots, x_n\rangle/I,$ where $I$ is the two sided ideal generated by  elements as in  $(iii)$ and $(iv)$
of the Definition \ref{def.skewpbwextensions}, i.e., elements of the form
\begin{equation}\label{rep.skewpbwext}
c_r + x_ir-c_{i,r}x_i,\quad r_0+r_1x_1+ \cdots + r_nx_n+ x_jx_i-c_{i,j}x_ix_j,
\end{equation}
where  $r\neq 0$, $c_r$, $c_{i,r}\neq 0$, $r_0$, $r_1,\dots, r_n$, $c_{i,j}\neq 0$ are elements in
$\mathbb{K}$, with $ 1\leq i,j\leq n$.

\item  If $A$ is pre-commutative, then $A = \mathbb{K}\langle x_1,\dots, x_n\rangle/I$, where $I$ is the two sided ideal generated by  elements  of the form
\begin{equation}\label{rep.pre-commut}
c_r + x_ir-c_{i,r}x_i,\quad r_1x_1+ \cdots + r_nx_n+ x_jx_i-c_{i,j}x_ix_j,
\end{equation}
 where  $r\neq 0$, $c_r$, $c_{i,r}\neq 0$,  $r_1,\dots, r_n$, $c_{i,j}\neq 0$ are elements in $\mathbb{K}$, with $ 1\leq i,j\leq n$.

 \item  If $A$ is constant, then $A = \mathbb{K}\langle x_1,\dots, x_n\rangle/I,$ where $I$ is the two sided ideal generated by  elements  of the form
\begin{equation}\label{rep.const}
r_0 +r_1x_1+ \cdots + r_nx_n+ x_jx_i-c_{i,j}x_ix_j,
\end{equation}
 where   $r_0, r_1,\dots, r_n$, $c_{i,j}\neq 0$ are elements in $\mathbb{K}$, with $ 1\leq i,j\leq n$.

\item  If $A$ is  quasi-commutative then $A = \mathbb{K}\langle x_1,\dots, x_n\rangle/I,$ where $I$ is the two sided ideal generated by  elements as in   $(iii')$ and $(iv')$ of the Definition \ref{def.quasicom}, i.e., elements of the form
\begin{equation}\label{rep.quasi-comm}
x_ir-c_{i,r}x_i,\quad  x_jx_i-c_{i,j}x_ix_j
\end{equation}
where  $r\neq 0$, $c_{i,r}\neq 0$, $c_{i,j}\neq 0$ are elements in $\mathbb{K}$, with $ 1\leq i,j\leq n$.
\item  If $A$ is  semi-commutative then $A = \mathbb{K}\langle x_1,\dots, x_n\rangle/I,$ where $I$ is the two sided ideal generated by  elements of the form
\begin{equation}\label{rep.semi-comm}
x_jx_i-c_{i,j}x_ix_j
\end{equation}
where  $c_{i,j}\neq 0$ are elements in $\mathbb{K}$, with $ 1\leq i,j\leq n$.

\end{enumerate}
\end{remark}

If otherwise is not assumed, in this paper all skew PBW extensions are
$\mathbb{K}$-algebras and extensions of the field $\mathbb{K}$ (i.e., $R=\mathbb{K}$ in Definition
\ref{def.skewpbwextensions}), so $A=\sigma(\mathbb{K})\langle x_1\dots, x_n\rangle$ is necessarily
a constant skew PBW extension.

\begin{proposition}\label{prop.pre-Koszul}

Let $A= \sigma(\mathbb{K})\langle x_1,\dots, x_n\rangle$ be a skew PBW extension. If $A$ is
pre-commutative then $A$ is pre-Koszul.

\end{proposition}

\begin{proof}

From (\ref{rep.pre-commut}) and (\ref{rep.const}) we have that
 $A = \mathbb{K}\langle x_1,\dots, x_n\rangle/I,$ where $I$ is the two sided ideal generated by  elements of the form
\begin{equation}\label{rep.preKskewpbwext}
r_1x_1+ \cdots +r_nx_n+ x_jx_i-c_{i,j}x_ix_j.
\end{equation}
Then we conclude that $A$ is pre-Koszul.
\end{proof}

\begin{example}\label{example3.5}
According to the classifications presented in the tables of Section \ref{section2}, the following
skew $PBW$ extensions are pre-Kozul algebras: classical polynomial ring over a field; particular
Sklyanin algebra; universal enveloping algebra of a Lie algebra; algebra of linear partial
q-dilation operators; additive analogue of the Weyl algebra; multiplicative analogue of the Weyl
algebra; quantum algebra $\mathcal{U}'(so(3,K))$; dispin algebra; Woronowicz algebra;
$q$-Heisenberg algebra; multi-parameter quantum affine $n$-space; types 1, 2, 3, 4, 5 and 6 of Sridharan enveloping
algebra of 3-dimensional Lie algebras.
\end{example}

\begin{proposition}\label{prop2}
Let $A$ be a skew PBW extension. If $A$ is semi-commutative then $A$ is pre-Koszul homogeneous.
\end{proposition}

\begin{proof}
If $A$  is a semi-commutative skew PBW extension then $A$ from (\ref{rep.semi-comm}) and
Proposition \ref{prop.pre-Koszul} we get that $A$ is pre-Kozul homogeneous.
\end{proof}

\begin{example}
From Example \ref{example3.5} we obtain the following examples of pre-Kozul homogeneous skew $PBW$
extensions: classical polynomial ring over a field; particular Sklyanin algebras; multiplicative
analogue of the Weyl algebra; multi-parameter quantum affine $n$-space; the Sridharan enveloping algebra of
3-dimensional Lie algebra with $[x,y]=[y,z]=[z,x]=0$.
\end{example}

Let $B$ be a pre-Koszul algebra. One can truncate the relations  in (\ref{rep.prekosz}) leaving only their homogeneous quadratic parts. Let $B^{(0)}$ be
the obtained  algebra. Then $B^{(0)}$ is called the \emph{associated homogeneous pre-Koszul algebra} of $B$. Note that $B$ is homogeneous if and only if $B^{(0)}\cong B$ as algebras.

\begin{proposition}\label{prop.asshompreKoszul}
Let $A$ be a pre-Koszul skew PBW extension, then $A^\sigma$ is the associated homogeneous pre-Koszul algebra of $A$.
\end{proposition}

\begin{proof}
Let $A$ be a  pre-Koszul skew PBW extension.  By Proposition \ref{remarkAsigma} there exists a quasi-commutative skew PBW extension $A^{\sigma}$ of $\mathbb{K}$ in $n$ variables $z_1,\dotsc, z_n$ defined by the relations $z_ir=c_{i,r}z_i$, $z_jz_i=c_{i,j}z_iz_j$, for $1\le i\le n$, where $c_{i,r}, c_{i,j}$ are the same constants that define $A$. Since $A$ is pre-Koszul then by Proposition \ref{prop.pre-Koszul} $A$ is constant and therefore  $A^{\sigma}$ is defined by the relations $z_jz_i=c_{i,j}z_iz_j$. Then $A^{(0)}\cong A^{\sigma}$.
\end{proof}

\subsection{Koszul algebras and skew PBW extensions}

Let $B$ be a finitely graded algebra generated in degree $1$; consider the \textit{Yoneda algebra}
of $B$ defined by
\begin{center}
$E(B):=\bigoplus_{i\geq 0}Ext_B^i(\mathbb{K},\mathbb{K})$;
\end{center}
the $Ext$ groups here are computed in the category of graded $B$-modules with graded $Hom$ spaces;
the product in $E(B)$ is defined in the following way: Let $\{P_i\xrightarrow{d_i} P_{i-1}\}_{i\geq
0}$ be a graded projective resolution of $\mathbb{K}$ that defines the groups
$Ext_B^i(\mathbb{K},\mathbb{K})$, with $P_{-1}:=\mathbb{K}$; moreover, let $\overline{f}\in
Ext_B^i(\mathbb{K},\mathbb{K})=\ker(d_{i+1}^*)/Im{f_i^*}$ with $f\in ker(d_{i+1}^*)\subseteq
Hom_B(P_i,\mathbb{K})$ and $\overline{g}\in
Ext_B^j(\mathbb{K},\mathbb{K})=\ker(d_{j+1}^*)/Im{f_j^*}$ with $g\in ker(d_{j+1}^*)\subseteq
Hom_B(P_j,\mathbb{K})$, then we define
\begin{align*}
Ext_B^i(\mathbb{K},\mathbb{K})\times Ext_B^j(\mathbb{K},\mathbb{K}) & \to Ext_B^{i+j}(\mathbb{K},\mathbb{K})\\
(\overline{f},\overline{g})& \mapsto \overline{f}\ \overline{g}:=\overline{fg'},
\end{align*}
where $g':P_{i+j}\to P_i$ is defined inductively by the following commutative diagrams:
\[
\begin{diagram}
\node[2]{P_j} \arrow{sw,t,..}{g_0} \arrow{s,r}{g} \\
\node{P_0} \arrow{e,b}{d_0} \node{\mathbb{K}}
\end{diagram}\Rightarrow
\begin{diagram}
\node[2]{P_{j+1}} \arrow{sw,t,..}{g_1} \arrow{s,r}{g_0 d_{j+1}} \\
\node{P_1} \arrow{e,b}{d_1} \node{Im(d_1)}
\end{diagram}\ \ \Rightarrow \cdots \ \Rightarrow\
\begin{diagram}
\node[2]{P_{j+i}} \arrow{sw,t,..}{g':=g_i} \arrow{s,r}{g_{i-1} d_{j+i}} \\
\node{P_i} \arrow{e,b}{d_i} \node{Im(d_i)}
\end{diagram}
\]
Can be proved that this product is well defined, i.e., it does not depend of the projective
resolution of $\mathbb{K}$ and the choosing of $g_0,g_1,\dots,g_{i-1}, g_i$; moreover, $fg'\in
\ker(d_{i+j+1}^*)$: In fact, from the step $i+1$ in the previous inductive procedure we have that
$d_{i+1}g_{i+1}=g_id_{i+j+1}$, so $fd_{i+1}g_{i+1}=fg_id_{i+j+1}$, i.e.,
$0=d_{i+1}^*(f)g_{i+1}=d_{i+j+1}^*(fg_i)$.

Thus, $E(B)$ is a graded algebra; note that the $\mathbb{K}$-vector space
$Ext_B^i(\mathbb{K},\mathbb{K})$ is graded
\begin{center}
$Ext_B^i(\mathbb{K},\mathbb{K})=\bigoplus_{j\geq 0}Ext_B^{i,j}(\mathbb{K},\mathbb{K})$,
\end{center}
with
\begin{center}
$Ext_B^{i,j}(\mathbb{K},\mathbb{K}):=(Ext_B^i(\mathbb{K},\mathbb{K}))_{-j}:=Ext_B^i(\mathbb{K},\mathbb{K}(-j))$,
\end{center}
so setting $E^{i,j}(B):=Ext_B^{i,j}(\mathbb{K},\mathbb{K})$ we get that
\begin{center}
$E(B)=\bigoplus_{i,j\geq 0}E^{i,j}(B)$
\end{center}
is a bigraded algebra. For $i\geq 0$, we write
\begin{center}
$E^i(B):=\bigoplus_{j\geq 0}E^{i,j}(B)$;
\end{center}
in particular,
\begin{center}
$E^0(B)=\bigoplus_{j\geq 0}Hom_B^j(\mathbb{K},\mathbb{K})=\bigoplus_{j\geq
0}(Hom_B(\mathbb{K},\mathbb{K}))_{-j}=\bigoplus_{j\geq 0}Hom_B(\mathbb{K},\mathbb{K}(-j))$,
\end{center}
with $Hom_B(\mathbb{K},\mathbb{K}(-j)):=\{f\in
Hom_B(\mathbb{K},\mathbb{K})|f(\mathbb{K}_l)\subseteq \mathbb{K}_{l-j}, l\in \mathbb{Z}\}$.

\begin{definition}\label{def.homKosz}
Let $B$ be a homogeneous pre-Koszul algebra, $B$  is called \textit{homogeneous Koszul} if the
following equivalent conditions hold:
\begin{enumerate}
\item[\rm (i)]  $Ext^{i,j}_B(\mathbb{K}, \mathbb{K})=0$ for $i\neq j$;
\item[\rm (ii)] $E(B)$ is generated by $E^{1,1}(B)$;
\item[\rm (iii)] The module $\mathbb{K}$ admits a \emph{linear free resolution}, i.e., a resolution by free $B$-modules
\[
\cdots \to P_2\to P_1\to P_0\to \mathbb{K}\to 0
\]
such that $P_i$ is generated in degree $i$.
\end{enumerate}
\end{definition}

\begin{definition}[\cite{Priddy1970}, Page 43]\label{def.Koszul}
 We say that a pre-Koszul algebra $B$ is a \emph{Koszul algebra} if $B^{(0)}$  is a homogeneous
Koszul algebra.
\end{definition}

\begin{remark}\label{Koszul impl homKos}
Notice that if $B$ is homogeneous Koszul algebra  then $B$ is Koszul. In fact, as $B$ is homogeneous then  $B^{(0)}\cong B$ as algebras and so $B^{(0)}$ is homogeneous Koszul, therefore $B$ is Koszul.
\end{remark}

 In the current literature, homogeneous Koszul algebras are called simply Koszul algebras. Some authors have studied Koszul algebras defined by Priddy in \cite{Priddy1970}. For example Koszul algebras are defined in \cite{Polishchuk}, analogous to the Definition \ref{def.Koszul}. Let $P\subseteq \mathbb{K}\bigoplus L_1 \bigoplus L_2$ a subspace of $F_2(L)$ and $A= L/\langle P\rangle$. Let $A^{(0)}= L/\langle R\rangle$, where $R$ is obtained by taking homogeneous part of $P$. $A$ is said to be (nonhomogeneous) Koszul if $A^{(0)}$ is homogeneous Koszul (see \cite{Polishchuk}, page 140). In \cite{LodayVallete2012} Koszul algebras are defined as follows. Let $V$ a graded vector space and a degree homogeneous subspace  $P\subseteq V\bigoplus V^{\otimes 2}$, the algebra $A=T(V)/\langle P\rangle$ is called (nonhomogeneous quadratic) Koszul if $P\cap V=\{0\}$, $\{P\otimes V + V\otimes P\}\cap V^{\otimes 2}\subseteq P\cap V^{\otimes 2}$ and $T(V)/\langle \pi(P) \rangle$ is homogeneous Koszul, where $\pi$ : $T(V)\twoheadrightarrow V^{\otimes 2}$ is the projection onto the quadratic part of the tensor algebra.
  R. Berger in \cite{Berger7} defined the notion of $N$-Koszul algebra, if $N = 2$, the  notion of homogeneous Koszul algebra is obtained. To avoid confusion, we still use the names given in the Definition \ref{def.homKosz} (homogeneous Koszul) and the Definition \ref{def.Koszul} (Koszul).\\

Let $L=\mathbb{K}\langle x_1,\dots, x_n\rangle$ the free associative algebra  in $n$ generators $x_1,\dots, x_n$. Let  $R$ a subspace of $F_2(L)=\mathbb{K}\bigoplus L_1\bigoplus L_2$, the algebra $L/\langle R\rangle$ is called (nonhomogeneous) \emph{quadratic algebra}. $L/\langle R\rangle$ is called \emph{homogeneous quadratic algebra} if $R$ is a subspace of $L_2$, for $\langle R\rangle$ the two-sided ideal of $L$ generated by $R$.  Let $A = \mathbb{K}\langle x_1,\dots, x_n\rangle/\langle R\rangle$ be a quadratic algebra with a fixed generators $\{x_1,\dots, x_n\}$. For a multiindex $\alpha:=(i_1,\dots, i_m)$, where $1\leq i_k\leq n$, we denote the monomials in $\mathbb{K}\langle x_1,\dots, x_n\rangle$ by  $x^{\alpha}:=x_{i_1}x_{i_2}\cdots x_{i_m}$. For $\alpha=\emptyset$ we set $x^{\emptyset}:=1$. Now let us equip the subspace $L_2$ with the basis consisting of the monomials $x_{i_1}x_{i_2}$. Let $S^{(1)}:=\{1, 2,\dots,n\}$, $S^{(1)}\times S^{(1)}$ the cartesian product, then for $R\subseteq L_2$ we obtain the set $S\subseteq S_1\times S_1$ of  pairs of indices $(l, m)$ for which the class of $x_lx_m$ in $L_2/R$ is not in the span of the classes of $x_rx_s$ with $(r, s) < (l, m)$, where $<$ denotes the lexicographical order (\cite{Polishchuk}, 4.1-Lemma 1.1). Hence, the relations in $A$ can be written in the following form:

\[
x_ix_j =\sum_{\substack{(r,s)<(i,j)\\(r,s)\in S}}c^{rs}_{ij} x_rx_s,\qquad (i, j) \in S^{(1)}\times S^{(1)}\ \backslash\ S.
\]

Define further $S^{(0)}:= \{\emptyset\}$,  and for $m\geq 2$,
\[
S^{(m)}:= \{(i_1,\dots, i_m) \mid (i_k, i_{k+1}) \in S, \ k = 1,\dots, m - 1\}
\]
and consider the monomials $\{x_{i_1}\cdots x_{i_m} \in A_m \mid  (i_1,\dots, i_m) \in S^{(m)}\}$. Note that these monomials always span $A_m$ as a vector space and and the monomials
\begin{equation}\label{eq.monPBW}
(A,S):=\{x_{i_1}\cdots x_{i_m} \mid  (i_1,\dots, i_m) \in \cup_{m> 0}S^{(m)}\}
\end{equation}
linearly span the entire $A$.  We call $(A,S)$ in (\ref{eq.monPBW}) a \emph{PBW-basis} of $A$ if they are linearly independent and hence form a $\mathbb{K}$-linear basis.  The elements $x_1, \dots, x_n$ are called PBW-\emph{generators} of $A$. A \emph{PBW-algebra} is a homogeneous quadratic algebra admitting a PBW-basis, i.e., there exists a permutation of $x_1,\dots, x_n$ such that the standard monomials in $x_1,\dots, x_n$ conform a $\mathbb{K}$-basis of $A$.

\begin{proposition}\label{prop.PBW alg}
Let $A$ be a semi-commutative skew PBW extension. Then $A$ is a PBW algebra.
 \end{proposition}

 \begin{proof}
  If $A=\sigma(\mathbb{K})\langle x_1,\dots, x_n\rangle$ is a semi-commutative skew PBW extension, then \linebreak $A=\mathbb{K}\langle x_1,\dots, x_n\rangle/\langle x_jx_i-c_{i,j}x_ix_j\rangle$ (as in  (\ref{rep.semi-comm})) is a homogeneous quadratic algebra with generators $x_1,\dots, x_n$ and relations  $x_jx_i-c_{i,j}x_ix_j$. Using the above notation we have that for $1\leq i \leq j\leq n$, the class of $x_ix_j$ is not in the span of the classes of $x_rx_s$ with $(r, s) < (i,j)$, but, the class of  $x_jx_i$ is in the span of the class of $x_ix_j$ with $(i,j)<(j,i)$. Therefore $S=\{(i,j)\mid 1\leq i \leq j\leq n\}=S^{(2)}$ and  $S^{(m)} = \{(i_1,\dots, i_m)\mid i_1\leq i_2\leq\cdots \leq i_m, \ 1\leq i_k \leq n\}$ for $m\geq 3$. Then
  \[
  (A,S)= \{x_1^{m_1}\cdots x_n^{m_n}\mid m_1,\dots, m_n\geq 0\}= {\rm Mon}(A):= \{x_1^{\alpha_1}\cdots
x_n^{\alpha_n}\mid (\alpha_1,\dots ,\alpha_n)\in
\mathbb{N}^n\}.
 \]

By Definition  \ref{def.skewpbwextensions} \rm{(ii)}, ${\rm Mon}(A)$ is a $\mathbb{K}$-basis  for $A$ and therefore $A$ is a PBW algebra.
\end{proof}

 \begin{theorem}[\cite{Priddy1970}, Theorem 5.3 ]\label{teo.PBW alg is Koszul}
 If $B$ is a PBW algebra then $B$ is a homogeneous Koszul algebra.
  \end{theorem}
The Theorem \ref{teo.PBW alg is Koszul} and this proof  can also be found in  \cite{Polishchuk}, Theorem 3.1, page 84; there they also present an example of a homogeneous Koszul algebra which is not PBW algebra.

 \begin{corollary}\label{cor. semi hom kosz}
 Every semi-commutative skew PBW extension is  homogeneous Koszul algebra.
 \end{corollary}

 \begin{proof}
 If following from Proposition \ref{prop.PBW alg} and Theorem \ref{teo.PBW alg is Koszul}.
 \end{proof}

\begin{theorem}\label{teo.Koszul sew PBW ext}
Every pre-commutative skew PBW extension is Koszul.
\end{theorem}

\begin{proof}

If $A$ is a pre-commutative skew PBW extension  then by Remark \ref{notac.skewpbwext},
$A= \mathbb{K}\langle x_1,\dots, x_n\rangle/I$, where $I$ is the two-sided ideal generated by relations of the form $x_jx_i-c_{i,j}x_ix_j+\sum_{t=1}^nk_tx_t$, $0 \neq c_{i,j}, k_t\in \mathbb{K}$, $1\leq i,j,t \leq n$. By Proposition \ref{prop.pre-Koszul} $A$ is pre-Koszul, therefore from   Proposition \ref{prop.asshompreKoszul}, $ A^{(0)}=A^{\sigma}=\mathbb{K}\langle x_1,\dots, x_n\rangle/\langle x_jx_i-c_{i,j}x_ix_j \rangle$ is the associated homogeneous pre-Koszul algebra of $A$. Note that $A^{\sigma}$ is semi-commutative, so by Corollary \ref{cor. semi hom kosz}, $A^{(0)}$ is a homogeneous Koszul algebra, i.e., $A$ is Koszul.
 \end{proof}

\begin{corollary}
If $A$ is a pre-commutative  skew PBW extension then $Gr(A)$ is homogeneous Koszul.
\end{corollary}

\begin{examples}\label{ex of homog koszul} Next we present some examples of homogeneous Koszul skew PBW extensions,
many of which had already been presented by other authors with the name of Koszul algebras. For
this purpose we use the classification given in Subsection \ref{ex and classif} and Corollary \ref{cor. semi hom kosz}: classical polynomial ring; particular Sklyanin algebra; Algebra of linear partial q-dilation ope\-ra\-tors; multiplicative analogue of the Weyl algebra;
multi-parameter quantum affine $n$-spaces; the Sridharan enveloping algebra of 3-dimensional Lie algebra with
$[x,y]=[y,z]=[z,x]=0$.
\end{examples}

\begin{examples}\label{ex of koszul} Recall that every homogeneous Koszul algebra is Koszul
(Remark \ref{Koszul impl homKos}), so Examples \ref{ex of homog koszul} are Koszul skew PBW extensions.
According to  classification given in Subsection \ref{ex and classif} and Theorem \ref{teo.Koszul sew PBW ext} the next skew PBW extensions are Koszul:
universal enveloping algebra of a Lie algebra, with $\mathbb{K}$ a field; diffusion algebra 1; quantum algebra;
Dispin algebra; Woronowicz algebra; $q$-Heisenberg algebra; types 1, 2, 3, 4, 5 and 6 of Sridharan enveloping algebra of 3-dimensional Lie algebra  (E\-xam\-ple \ref{sridharan}).
\end{examples}

Note that some particular classes of skew PBW extensions in  Examples \ref{ex of homog koszul} and \ref{ex of koszul} represent the same algebra. For example, Sridharan enveloping algebra of 3-dimensional Lie algebra of type 1 and the classical polynomial ring $\mathbb{K}[x,y,z]$ are the same algebra.

\section{PBW deformations}\label{sect.PBW defor}

Let $V$ be a vector space over a field $\mathbb{K}$ and let $T(V) =\bigoplus T^i(V)$ be its tensor algebra over $\mathbb{K}$. Consider the natural filtration $F_i(T)=\{\bigoplus T^j(V)\mid j\leq i\}$ of $T(V)$. Fix a subspace $P\subseteq F_2(T)=\mathbb{K}\bigoplus V\bigoplus (V\otimes V)$, and let us consider the two-sided
ideal $\langle P\rangle$ in $T(V)$  generated by $P$. Let $A=T(V)/\langle P\rangle$ be a nonhomogeneous quadratic algebra. It inherits a filtration $A_0\subseteq A_1 \subseteq\cdots A_n \subseteq\cdots$ from $T(V)$, let  $Gr(A)$  the associated graded algebra. Consider the natural projection $\pi : F_2 (T)=\mathbb{K}\bigoplus V\bigoplus (V\otimes V)\to V\otimes V$ on the homogeneous component, set $R=\pi(P)$ and consider the homogeneous quadratic algebra $T(V)/\langle R\rangle$. $T(V)/\langle R\rangle$ is called the \emph{homogeneous version}  (or the \emph{induced
homogeneous quadratic}) algebra of $A$ determined by $P$.  We have the natural epimorphism $p: T(V)/\langle R\rangle \to Gr(A)$ (induced by the projection $T(V)\to A$).

\begin{definition}[\cite{Braverman}, Page 316]\label{def.PBW deformation} With the above notation, a nonhomogeneous quadratic algebra $A:=T(V)/\langle P\rangle$ is a \emph{Poincar\'e-Birkhoff-Witt {\rm (PBW)} deformation}  (or satisfies the PBW \emph{property} with respect to the subspace $P$ of $F_2(T)$)  of $B:=T(V)/\langle R\rangle$ if the natural projection $p: T(V)/\langle R\rangle)\to Gr (A)$ is an isomorphism.
\end{definition}

\begin{proposition}[\cite{LodayVallete2012}, Theorem 3.6.4]
 Let $A=T(V)/\langle P\rangle$ with $P\subseteq V\bigoplus V^{\otimes 2}$, $P\cap V =\{0\}$ and $\{ P\otimes V + V\otimes P\}\cap V^{\otimes 2}\subseteq P\cap V^{\otimes 2}$. If $A$ is Koszul, then the epimorphism $p$: $B=T(V)/\langle R \rangle\to Gr(A)$ is an isomorphism of graded algebras, i.e., $A$ is a PBW deformation of $B$.
\end{proposition}

\begin{corollary}[\cite{LodayVallete2012}, Corollary 3.6.5] Let $A=T(V)/\langle P\rangle$, with $P\subseteq V\bigoplus V^{\otimes 2}$.
If $B= T(V)/\langle R\rangle$ is homogeneous Koszul, then:

\begin{enumerate}
\item $P\cap V =\{0\}\Leftrightarrow \langle P \rangle \cap V =\{0\}$
\item $\{ P\otimes V + V\otimes P\}\cap V^{\otimes 2}\subseteq P\cap V^{\otimes 2} \Leftrightarrow \langle P \rangle \cap \{V\bigoplus V^{\otimes 2}\}=P$.
\end{enumerate}
\end{corollary}

\begin{remark}\label{homogeneous version and const}  \cite{Braverman}, Lemma 0.4  establishes that if the algebra $T(V)/\langle P\rangle$ is a PBW deformation of $T(V)/\langle R\rangle$ then
it satisfies the following conditions:
\begin{enumerate}
\item[\rm (I)] $P\cap F_1(T) = 0$;
\item[\rm (J)] $(F_1(T)\cdot P\cdot F_1(T))\cap F_2(T) = P$.
\end{enumerate}
If a nonhomogeneous quadratic algebras  satisfy $(I)$ then the subspace $P\subset F_2(T)$ can be described in terms of two maps $\alpha: R\to V$ and $\beta: R\to \mathbb{K}$ as $P=\{ x-\alpha(x)-\beta(x)\mid x\in R\}$. If $A=T(V)/\langle P\rangle$ is a  PBW deformations of its homogeneous version then $P$ can not have relations in $F_1(T)$, so:
 \begin{enumerate}
 \item[\rm (i)] If $A$ is a  PBW deformation of some skew PBW extension $B$, then $A$ is constant.
 \item[\rm(ii)] The homogeneous version of a skew PBW extension $A$ is the  skew PBW extension $B$ such that the conditions $(i)$ and $(ii)$ of Definition \ref{def.skewpbwextensions} for $A$ are satisfy for $B$, and the conditions $(iii)$ and $(iv)$ are replaced by $x_jx_i-c_{i,j}x_ix_j=0$, where $c_{i,j}$ are the same that for $A$.
 \item[\rm(iii)]  The homogeneous version of a skew PBW extension is homogeneous Koszul.
 \end{enumerate}
 \end{remark}

 For example the homogeneous version for  the universal enveloping algebra of a Lie algebra $\mathcal{G}$, $\mathcal{U}(\mathcal{G})$ is the symmetric algebra $\mathbb{S}(\mathcal{G})$.

\begin{proposition}\label{prop.deformation of skew} Let $A$ be a constant skew PBW extension of a field $\mathbb{K}$. Then $A$ is a PBW deformation of its homogeneous version $B$.
\end{proposition}
\begin{proof}
Let $A$ be a constant skew PBW extension of a field $\mathbb{K}$ then, $x_jx_i-c_{i,j}x_ix_j+r_0+r_1x_1+\cdots +r_nx_n$ (as in Definition \ref{def.skewpbwextensions}) are the generated relations of the subspace $P$, that is, $A=\mathbb{K}\langle x_1, \dots, x_n\rangle/\langle P \rangle$. Then the subspace $\pi(P) =R$ is generated by the relations $x_jx_i-c_{i,j}x_ix_j$, i.e, $\mathbb{K}\langle x_1, \dots, x_n\rangle/\langle R \rangle=B$ is the  homogeneous version of $A$. Now for Theorem \ref{teo.Gr(A)}, $Gr(A)\cong A^{\sigma}$ where $A^{\sigma}$ is a skew PBW extension of $\mathbb{K}$ in  $n$ variables $z_1,\dotsc, z_n$ defined by the relations  $z_jz_i=c_{i,j}z_iz_j$, for $1\le i\le n$. So by Remark \ref{homogeneous version and const},  $A^{\sigma}\cong B$ and therefore $Gr(A)\cong B$, i.e., $A$ is a PBW deformation of $B$.
\end{proof}

Note that if a skew PBW extension $A$ is not constant then the Proposition \ref{prop.deformation of skew} fails, indeed: the homogeneous version of $A$ is the skew PBW extension $B$ with relations $x_jx_i-c_{i,j}x_ix_j=0$, where $c_{i,j}$ are the same that for $A$, but $Gr(A)$ is defined defined by the relations $z_ir=c_{i,r}z_i$, $z_jz_i=c_{i,j}z_iz_j$ (see Theorem \ref{teo.Gr(A)} and Proposition \ref{remarkAsigma}), so  $Gr(A)\ncong B$.\\

Let $T(V)/\langle P\rangle$ be a nonhomogeneous
quadratic algebra. Take $R=p(P)\subseteq T^2(V)$ and consider the corresponding
homogeneous quadratic algebra $A=T(V)/\langle R\rangle$. The main theorem of \cite{Braverman} establishes that if $A$ is a homogeneous Koszul algebra then conditions $(I)$ and $(J)$ in Remark \ref{homogeneous version and const} imply that the  algebra $T(V)/\langle P\rangle$ is a PBW deformation of $A$.

\begin{proposition}
If $A$ is a PBW deformation of a skew PBW extension $B$, then $B$ is homogeneous Koszul.
\end{proposition}

\begin{proof}
Let $A=\sigma(\mathbb{K})\langle x_1,\dots, x_n\rangle$ be a PBW deformation of $B$, then   $$A=\mathbb{K}\langle x_1,\dots, x_n\rangle/\langle x_jx_i-c_{i,j}x_ix_j+ k_0+k_1x_1+\cdots + k_nx_n\rangle,$$ with $c_{i,j}\in \mathbb{K} \ \backslash \ \{0\}$, $k_l\in \mathbb{K}$, $1\leq i,j\leq n$, $0\leq l\leq n$ and $B\cong \mathbb{K}\langle x_1,\dots, x_n\rangle/\langle x_jx_i-c_{i,j}x_ix_j\rangle$. Then $B$ is a semi-commutative skew PBW extension of $\mathbb{K}$, and by Corollary \ref{cor. semi hom kosz}, we conclude that $B$ is homogeneous Koszul.
\end{proof}

\begin{center}
\textbf{Acknowledgements}
\end{center}
The authors express their gratitude to Professor Oswaldo Lezama for valuable suggestions and
helpful comments for the improvement of the paper.


\begin{thebibliography}{30}

\bibitem{lezamaore} J.P. Acosta, C. Chaparro, O. Lezama, I. Ojeda and C. Venegas,
Ore and Goldie theorems for skew $PBW$ extensions, {\em Asian-European J. Math.} \textbf{06} (2013) 1350061
[20 pages].

\bibitem{ArtamonovDerivations} V. Artamonov, Derivations of Skew PBW-Extensions, {\em Commun. Math. Stat.} \textbf{3}(4) (2015) 449-457.

\bibitem{Backelin} J. Backelin and R. Fr\"oberg, Koszul algebras, Veronese subrings and rings with
linear resolutions, {\em Rev. Roumaine Math. Pures Appl.} \textbf{30}(2) (1985)  85-97.

\bibitem{BeilinsonGinzburgSoerge1996} A. Beilinson, V. Ginzburg and W. Soergel, Koszul duality patterns in representation theory, {\em J. Am. Math. Soc}. \textbf{9} (1996)  473-527.


\bibitem{Berger7} R. Berger, Koszulity for nonquadratic algebras, {\em J. Algebra}
\textbf{239}  (2001) 705-734.


\bibitem{Braverman} A. Braverman and D. Gaitsgory,  Poincar\'e-Birkhoff-Witt theorem for quadratic algebras of Koszul type, {\em J. Algebra} \textbf{181}  (1996) 315-328.

\bibitem{Cassidy2008} T. Cassidy and B.  Shelton, Generalizing the notion of Koszul algebra,
 {\em Math. Z}. \textbf{260} (2008) 93-114.

\bibitem{LezamaGallego} C. Gallego and O. Lezama, Gr\"obner bases for ideals of $\sigma$-PBW extensions, {\em Comm.  Algebra} \textbf{39}(1) (2011) 50-75.

\bibitem{Gallego4} C. Gallego and O. Lezama, $d$-Hermite rings and skew PBW extensions, {\em S\~{a}o Paulo J. Math. Sci.} \textbf{10}(1) (2016) 60-72.


\bibitem{Jimenez2} H. Jim\'enez and O. Lezama, Gr\"obner bases for modules over
$\sigma-PBW$ extensions, {\em Acta Math. Academiae Paedagogicae Ny\'{\i}regyh\'aziensis} \textbf{32}(1) (2016)
39-66.

\bibitem{LezamaAcostaReyes2015} O. Lezama, J. P. Acosta and A. Reyes, Prime ideals of skew PBW extensions, {\em Rev. Un. Mat. Argentina} \textbf{56}(2) (2015) 39-55.

\bibitem{LezamaReyes} O. Lezama and A. Reyes, Some homological properties of skew PBW extensions, {\em  Comm. Algebra} \textbf{42} (2014) 1200-1230.

\bibitem{LiH2012} H. Li, {\em Gr\"obner Bases in Ring Theory}, World Scientific Publishing Company, 2012.

\bibitem{Li2012} L. Li, A generalized Koszul theory and its application, {\em Trans. Amer. Math. Soc.} \textbf{366}(2) (2014) 931-977.

\bibitem{LodayVallete2012}  J-L. Loday and B. Vallette, {\em Algebraic Operads},  Grundlehren Math. Wiss, Vol. \textbf{ 346} (Springer, Heidelberg, 2012).

\bibitem{Nuss1}  P. Nuss,  L'homologie cyclique des algèbres enveloppantes des algèbres de Lie de dimension trois, {\em J. Pure Appl. Algebra} \textbf{73} (1991) 39-71.

\bibitem{Polishchuk} A. Polishchuk and C. Positselski,   {\em Quadratic Algebras}, Univ. Lecture Ser., Vol. \textbf{37}  (Amer. Math. Soc., Providence, RI, 2005).

\bibitem{Priddy1970} S. Priddy, Koszul resolutions, {\em Trans. Am. Math. Soc.} \textbf{152} (1970) 39-60.

\bibitem{Reyes2013} A. Reyes,  Gelfand-Kirillov dimension of skew PBW extensions, {\em Rev. Col. Mat.} \textbf{47}(1) (2013) 95-111.

\bibitem{Reyes2014} A. Reyes, Uniform dimension over skew PBW extensions, {\em Rev. Col. Mat.} \textbf{48}(1) (2014) 79-96.

\bibitem{Reyes2014UIS} A. Reyes, Jacobson's conjecture and skew PBW extensions, {\em Rev. Integr. Temas Mat.} \textbf{32}(2) (2014) 139-152.

\bibitem{Reyes2015} A. Reyes, Skew PBW extensions of Baer, quasi-Baer, p.p. and p.q.-rings, {\em Rev. Integr. Temas Mat.} \textbf{33}(2) (2015) 173-189.

\bibitem{Rogalski} D. Rogalski,  An introduction to non-commutative projective algebraic geometry,
{\em arXiv:1403.3065 [math.RA]}.


\bibitem{Smith2} S.P. Smith, Some finite dimensional algebras related to elliptic curves, {\em  Rep. Theory
of Algebras and Related topics } CMS Conf. Proc. \textbf{19},  AMS (1996) 315-348.

\bibitem{SuarezLezamaReyes2015}H. Su\'arez, O. Lezama and A. Reyes, Some relations between $N$-Koszul, Artin-Schelter regular and Calabi-Yau algebras with skew PBW extensions, {\em Revista Ciencia en Desarrollo} \textbf{6}(2) (2015) 205-213.

\bibitem{Venegas2015} C. Venegas, Automorphisms for skew PBW extensions and skew quantum polynomial rings, {\em Comm.  Algebra}, \textbf{43}(5) (2015) 1877-1897.

\bibitem{Woodcock1998} D. Woodcock, Cohen-Macaulay complexes and Koszul rings, {\em J. Lond. Math. Soc}. \textbf{57} (1998) 398-410.

\end{thebibliography}
\end{document}